\documentclass[9pt]{article}
\usepackage{mathrsfs}
\usepackage{amsthm}
\usepackage{amssymb}
\usepackage{amsmath}
\usepackage{graphicx}
\usepackage{color}
\usepackage{amsfonts}
\usepackage{float}
\usepackage{cite}
\usepackage [latin1]{inputenc}
\usepackage[text={140mm,210mm},left=35mm,vmarginratio=1:1]{geometry}
\newtheorem{theorem}{Theorem}[section]

\newtheorem{lemma}[theorem]{Lemma}

\numberwithin{equation}{section}
\normalsize

\begin{document}
\title{\textbf{Inhomogeneous central limit theorems for the voter model occupation times}}

\author{Xiaofeng Xue \thanks{\textbf{E-mail}: xfxue@bjtu.edu.cn \textbf{Address}: School of Mathematics and Statistics, Beijing Jiaotong University, Beijing 100044, China.}\\ Beijing Jiaotong University}

\date{}
\maketitle

\noindent {\bf Abstract:} In this paper, we extend the functional central limit theorems for the occupation times of the voter models on lattices given in \cite{Xue2026} to the case where the initial distribution is a spatially inhomogeneous product measure. The duality relationship between the voter model and the coalescing random walk and the Donsker's invariance principle of the simple random walk play the key roles in the proofs of our main results.  

\quad



\section{Introduction}\label{section one}
This paper is a further investigation of the topic discussed in \cite{Xue2026}. We extend the functional central limit theorems of the voter model occupation times on lattices to the case where the initial distribution is a product measure without spatial homogeneity. For later use, we first introduce some notations. For each integer $d\geq 1$, we denote by $\mathbb{Z}^d$ the $d$-dimensional lattice. For any $x=(x_1, x_2, \ldots, x_d)\in \mathbb{Z}^d$ and $y=(y_1, y_2, \ldots, y_d)\in \mathbb{Z}^d$, we write $x\sim y$ when and only when
\[
\sum_{i=1}^d|x_i-y_i|=1,
\]
i.e., $y$ is one of $2d$ neighbors of $x$. We denote by $\mathbf{0}$ the origin of $\mathbb{Z}^d$, i.e., $\mathbf{0}=(0, 0, \ldots, 0)$. Now we recall the definition of the voter model on $\mathbb{Z}^d$. In the voter model, each $x\in \mathbb{Z}^d$ is a supporter of  one of two opposite opinions, namely `$1$' and `$0$', of a topic. For any $x\in \mathbb{Z}^d$ and $t\geq 0$, let $\eta_t(x)$ be the opinion of $x$ at moment $t$, then
\[
\eta_t=\{\eta_t(x)\}_{x\in \mathbb{Z}^d}
\]
is a configuration in $\{0, 1\}^{\mathbb{Z}^d}$. The voter model evolves as follows. For any $x\in \mathbb{Z}^d$ and $y\sim x$, $x$ adopts the opinion of $y$ at rate $1$. Consequently, the voter model $\{\eta_t\}_{t\geq 0}$ on $\mathbb{Z}^d$ is a continuous-time Markov process with generator $\mathcal{L}$ given by
\begin{equation}\label{equ 1.1 generator}
\mathcal{L}f(\eta)=\sum_{x\in \mathbb{Z}^d}\sum_{y: y\sim x}\left(f(\eta^{x, y})-f(\eta)\right)
\end{equation}
for any $\eta\in \{0, 1\}^{\mathbb{Z}^d}$ and local $f: \{0, 1\}^{\mathbb{Z}^d}\rightarrow \mathbb{R}$, where $\eta^{x, y}\in \{0, 1\}^{\mathbb{Z}^d}$ is defined as
\[
\eta^{x, y}(z)=
\begin{cases}
\eta(y) & \text{~if~}z=x,\\
\eta(z) & \text{~if~}z\neq x
\end{cases}
\]
for all $z\in \mathbb{Z}^d$. Readers could resort to Chapter 5 of \cite{Lig1985} and Part {\rm \uppercase\expandafter{\romannumeral2}} of \cite{Lig1999} for a detailed review of the basic properties of the voter model.

For any $\eta\in\{0, 1\}^{\mathbb{Z}^d}$, we denote by $\mathbb{P}_\eta$ the probability measure of $\{\eta_t\}_{t\geq 0}$ starting from $\eta$. For any probability measure $\mu$ on $\{0, 1\}^{\mathbb{Z}^d}$, we denote by $\mathbb{P}_\mu$ the probability measure of $\{\eta_t\}_{t\geq 0}$ with initial distribution $\mu$, i.e.,
\[
\mathbb{P}_\mu(\cdot)=\int_{\{0, 1\}^{\mathbb{Z}^d}} \mathbb{P}_\eta(\cdot)\mu(d\eta).
\]
We denote by $\mathbb{E}_\mu$ (resp. $\mathbb{E}_\eta$) the expectation with respect to $\mathbb{P}_\mu$ (resp. $\mathbb{P}_\eta$).

Throughout this paper, we assume that $\rho$ is a uniformly continuous function on $\mathbb{R}^d$ such that $0<\rho(u)<1$ for any $u\in \mathbb{R}^d$. For any integer $N\geq 1$, we further denote by $\nu_{\rho, N}$ the product measure on $\{0, 1\}^{\mathbb{Z}^d}$ under which $\{\eta(x)\}_{x\in \mathbb{Z}^d}$ are independent and
\[
\nu_{\rho, N}\left(\eta(x)=1\right)=\rho\left(x/\sqrt{N}\right)
\]
for all $x\in \mathbb{Z}^d$. In this paper, we are concerned with the occupation times of the voter model with initial distribution $\nu_{\rho, N}$. In detail, when $\eta_0$ is distributed with $\nu_{\rho, N}$, the centered occupation time process $\{\xi_t^N\}_{t\geq 0}$ on $\mathbf{0}$ is defined as
\begin{equation}\label{equ 1.2 centered occupation time}
\xi_t^N=\int_0^t \left(\eta_s(\mathbf{0})-\mathbb{E}_{\nu_{\rho, N}}\eta_s(\mathbf{0})\right)ds
\end{equation}
for any $t\geq 0$. We write $\xi_t^N$ as $\xi_t^{N, d}$ when we need to emphasize the dependence of $\xi_t^N$ on the dimension $d$. In the homogeneous case where $\rho\equiv p\in (0, 1)$, it is shown in \cite{Cox1983} that, for any $d\geq 2$ and $t>0$, $\frac{1}{h_d(N)}\xi_{tN}^{N, d}$ converges weakly to a Gaussian distribution as $N\rightarrow+\infty$, where
\begin{equation}\label{equ 1.3 hdt}
h_d(t)=
\begin{cases}
\frac{t}{\sqrt{\log t}} & \text{~if~} d=2,\\
t^{3/4} & \text{~if~} d=3,\\
\sqrt{t\log t} & \text{~if~} d=4,\\
\sqrt{t} & \text{~if~} d\geq 5
\end{cases}
\end{equation}
for any $t\geq 0$. Reference \cite{Xue2026} extends the above CLT to a functional version when $d\geq 3$. It is shown in \cite{Xue2026} that, when $\rho\equiv p\in (0, 1)$, for any $d\geq 3$ and $T>0$,
\[
\left\{\frac{1}{h_d(N)}\xi_{tN}^{N, d}:~0\leq t\leq T\right\}
\]
converges weakly, with respect to the uniform topology of $C[0, T]$, to $\{V_t\}_{0\leq t\leq T}$ as $N\rightarrow+\infty$, where $V_\cdot$ is the Brownian motion times a constant when $d\geq 4$ or a Gaussian process without independent increments when $d=3$. In this paper, we extend the aforesaid functional CLT given in \cite{Xue2026} to the case where $\rho$ is more general than a constant. For the precise statements of our main results, see Section \ref{section two}.

Central limit theorems of the occupation times of the interacting particle systems starting from a spatially inhomogeneous product measure have been discussed for the simple exclusion process (SEP). In the SEP, each vertex is vacant or occupied by a particle. All these particles perform independent simple random walks except that any jump to an occupied vertex is suppressed. It is shown in \cite{Xu2025} that, when $d\geq 2$ and the initial distribution of the SEP on $\mathbb{Z}^d$ is $\nu_{\rho, N}$, the centered occupation time process of the SEP on $\mathbf{0}$ converges weakly, with respect to the uniform topology, to the It\^{o} integral of a function with respect to a standard Brownian motion under proper time-space scaling, where the aforesaid function is related to the transition probabilities of the simple random walk. The main result of this paper in the case $d\geq 4$ is an analogue of the above conclusion given in \cite{Xu2025}. For mathematical details, see Section \ref{section two}.

\section{Main results}\label{section two}
In this section, we give the main results of this paper. We first introduce some notations for later use. We denote by $\{X_t\}_{t\geq 0}$ the continuous-time simple random walk on $\mathbb{Z}^d$ with generator $\mathcal{G}$ defined as
\[
\mathcal{G}g(x)=\sum_{y: y\sim x}\left(g(y)-g(x)\right)
\]
for any $x\in \mathbb{Z}^d$ and bounded $g: \mathbb{Z}^d\rightarrow \mathbb{R}$. We denote by $\hat{\mathbb{P}}$ the probability measure of $\{X_t\}_{t\geq 0}$ and by $\hat{\mathbb{E}}$ the expectation with respect to $\hat{\mathbb{P}}$.

We denote by $\{p_t\}_{t\geq 0}$ the transition probabilities of $\{X_t\}_{t\geq 0}$, i.e.,
\[
p_t(x, y)=\hat{\mathbb{P}}\left(X_t=y\big|X_0=x\right)
\]
for any $x, y\in \mathbb{Z}^d$ and $t\geq 0$. We denote by $\gamma$ the term
\[
\hat{\mathbb{P}}\left(X_t\neq \mathbf{0}\text{~for all~}t\geq 0\big|X_0\sim \mathbf{0}\right).
\]
When we need to emphasize the dependence of $p_t, \gamma$ on $d$, we write $p_t, \gamma$ as $p_{t, d}, \gamma_d$ respectively.

We denote by $\{B_t\}_{t\geq 0}$ the real-valued standard Brownian motion starting from $0$. For each $d\geq 2$, we denote by $\{\mathcal{W}_t^d\}_{t\geq 0}$ the $\mathbb{R}^d$-valued standard Brownian motion starting from $\mathbf{0}$, i.e.,
\[
\mathcal{W}_t^d=\left(B_t^1, B_t^2, \ldots, B_t^d\right),
\]
where $\{B_t^1\}_{t\geq 0}, \ldots, \{B_t^d\}_{t\geq 0}$ are independent copies of $\{B_t\}_{t\geq 0}$. For any $t\geq 0$ and $u\in \mathbb{R}^d$, we define
\[
\varrho(t, u)=\mathbb{E}\rho\left(\mathcal{W}_{2t}^d+u\right),
\]
where $\rho$ is defined as in Section \ref{section one}. According to the Kolmogorov-Chapman equation, $\varrho$ satisfies that
\[
\begin{cases}
&\frac{d}{dt}\varrho(t, u)=\Delta \varrho(t, u),\\
&\varrho(0, \cdot)=\rho(\cdot),
\end{cases}
\]
i.e., $\varrho$ is a solution to the heat equation.

Now we give our main result in the case $d\geq 4$.
\begin{theorem}\label{theorem 2.1 main d geq 4}
Let $d\geq 4, T>0$, $\rho$ be defined as in Section \ref{section one} and $h_d$ be defined as in \eqref{equ 1.3 hdt}. If $\eta_0$ is distributed with $\nu_{\rho, N}$, then
$
\left\{\frac{1}{h_d(N)}\xi_{tN}^{N, d}:~0\leq t\leq T\right\}
$
converges weakly, with respect to the uniform topology of $C[0, T]$, to
\[
\left\{\int_0^tA_{s, d}dB_s:~0\leq t\leq T\right\}
\]
as $N\rightarrow+\infty$, where $\xi_t^{N, d}$ is defined as in \eqref{equ 1.2 centered occupation time} and
\[
A_{s, d}=
\begin{cases}
\sqrt{\frac{\gamma_4\varrho(s, \mathbf{0})\left(1-\varrho(s, \mathbf{0})\right)}{\pi^2}} & \text{~if~}d=4,\\
\sqrt{4d\gamma_d\left(\int_0^{+\infty}\theta p_{\theta, d}\left(\mathbf{0}, \mathbf{0}\right)d\theta\right)\varrho(s, \mathbf{0})\left(1-\varrho(s, \mathbf{0})\right)} & \text{~if~}d\geq 5.
\end{cases}
\]
\end{theorem}
Note that, according to the local central limit theorem of the simple random walk on $\mathbb{Z}^d$, $p_t(\mathbf{0}, \mathbf{0})=\Theta(t^{-d/2})$ for sufficiently large $t$ and hence $\int_0^{+\infty}\theta p_{\theta, d}\left(\mathbf{0}, \mathbf{0}\right)d\theta<+\infty$ when and only when $d\geq 5$. When $\rho\equiv p\in (0, 1)$, $A_{s, d}$ is a constant function and then Theorem \ref{theorem 2.1 main d geq 4} reduces to the main theorem of \cite{Xue2026} in the case $d\geq 4$.

To introduce the main result in the case $d=3$, we further define
\[
b_t(s, u)=\int_0^{t-s}\frac{1}{(4\pi r)^{\frac{3}{2}}}e^{-\frac{\|u\|_2^2}{4r}}dr
\]
for any $0\leq s\leq t$ and $u\in \mathbb{R}^3$, where
$
\|u\|_2^2=\sum_{i=1}^3u_i^2
$
for any $u=(u_1, u_2, u_3)\in \mathbb{R}^3$. Then, we denote by $\{\zeta_t\}_{t\geq 0}$ the Gaussian process with continuous sample path, mean zero and covariance functions given by
\[
{\rm Cov}\left(\zeta_{t_1}, \zeta_{t_2}\right)=\int_0^{t_1}\left(\int_{\mathbb{R}^3}b_{t_1}(s, u)b_{t_2}(s, u)\varrho(s, u)\left(1-\varrho(s, u)\right)du\right)ds
\]
for any $0\leq t_1\leq t_2$.

Now we give our main result in the case $d=3$.
\begin{theorem}\label{theorem 2.2 main d=3}
Let $d=3, T>0$, $\rho$ be defined as in Section \ref{section one}. If $\eta_0$ is distributed with $\nu_{\rho, N}$, then
$
\left\{\frac{1}{N^{\frac{3}{4}}}\xi_{tN}^{N, 3}:~0\leq t\leq T\right\}
$
converges weakly, with respect to the uniform topology of $C[0, T]$, to
$
\left\{\sqrt{12\gamma_3}\zeta_t:~0\leq t\leq T\right\}
$
as $N\rightarrow+\infty$, where $\xi_t^{N, 3}$ is defined as in \eqref{equ 1.2 centered occupation time}.
\end{theorem}

According to the fact that
\[
\int_{\mathbb{R}^3}\frac{1}{(4\pi r_1)^{\frac{3}{2}}}e^{-\frac{\|u\|_2^2}{4r_1}}\frac{1}{(4\pi r_2)^{\frac{3}{2}}}e^{-\frac{\|u\|_2^2}{4r_2}} du
=\frac{1}{(4\pi (r_1+r_2))^{\frac{3}{2}}},
\]
it is easy to check that Theorem \ref{theorem 2.2 main d=3} reduces to the main theorem of \cite{Xue2026} in the case $d=3$ when $\rho\equiv p\in (0, 1)$.

Proofs of Theorems \ref{theorem 2.1 main d geq 4} and \ref{theorem 2.2 main d=3} utilize a strategy similar to that given in \cite{Xue2026} to deal with the homogeneous case, where the martingale decomposition method introduced in \cite{Kipnis1987} and the Poisson flow method introduced in \cite{Birkner2007} are applied. The main new difficulty in the inhomogeneous case is to estimate the term $\mathbb{E}_{\nu_{\rho, N}}\left(\left(\eta_{sN}(x)-\eta_{sN}(y)\right)^2\right)$ for $x\sim y$ and large $N$. This term is generated from a quadratic variation computing in the Poisson flow method. In the homogeneous case where $\rho\equiv p\in (0, 1)$, $\mathbb{E}_{\nu_{\rho, N}}\left(\left(\eta_{sN}(x)-\eta_{sN}(y)\right)^2\right)$ does not depend on the choice of $x$ and it is easy to show that
\[
\lim_{N\rightarrow +\infty}\mathbb{E}\left(\left(\eta_{sN}(x)-\eta_{sN}(y)\right)^2\right)=2\gamma_dp(1-p)
\]
via the duality relationship between the voter model and the coalescing random walk. In the inhomogeneous case, via the Donsker's invariance principle for the simple random walk, we will show that
\[
\mathbb{E}_{\nu_{\rho, N}}\left(\left(\eta_{sN}(x)-\eta_{sN}(y)\right)^2\right)\approx 2\gamma_d\varrho(s, x/N)\left(1-\varrho(s, x/N)\right)
\]
for large $N$. For mathematical details, see Sections \ref{section three} and \ref{section four}.

\section{Proof of Theorem \ref{theorem 2.1 main d geq 4}}\label{section three}
In this section we prove Theorem \ref{theorem 2.1 main d geq 4}. As we have introduced in Section \ref{section two}, the strategy of our proof is similar to that given in \cite{Xue2026} to deal with the homogeneous case. So we only highlight the main new difficulties generated by the inhomogeneous assumption. We omit details which do not depend on the spatial inhomogeneity and are nearly the same as that in the homogeneous case.

\subsection{Proof of Theorem \ref{theorem 2.1 main d geq 4}: The case $d\geq 5$}\label{subsection 3.1}
In this subsection we prove Theorem \ref{theorem 2.1 main d geq 4} in the case $d\geq 5$. So throughout this subsection we assume that $d\geq 5$. For each $N\geq 1$ and $x\in \mathbb{Z}^d$, we define
\[
g_N(x)=\int_0^{+\infty}e^{-s/N}p_s(\mathbf{0}, x)ds.
\]
For each $N\geq 1, t\geq 0$ and $\eta\in \{0, 1\}^{\mathbb{Z}^d}$, we further define
\[
G_N(t, \eta)=\sum_{x\in \mathbb{Z}^d}\left(\eta(x)-\mathbb{E}_{\nu_{\rho, N}}\eta_t(\mathbf{x})\right)g_N(x)
\]
and
\[
M_t^N=G_N(t, \eta_t)-G_N(0, \eta_0)-\int_0^t\left(\partial_s+\mathcal{L}\right)G_N(s, \eta_s)ds,
\]
where $\mathcal{L}$ is defined as in \eqref{equ 1.1 generator}. Then $\{M_t^N\}_{t\geq 0}$ is a martingale according to the Dynkin's martingale formula. By Kolmogorov-Chapman equation,
\[
\partial_s \mathbb{E}_{\nu_{\rho, N}}\eta_s(x)=\sum_{y:y\sim x}\left(\mathbb{E}_{\nu_{\rho, N}}\eta_s(y)-\mathbb{E}_{\nu_{\rho, N}}\eta_s(x)\right).
\]
Then, according to the expression of $\mathcal{L}$ and the integration-by-parts formula, it is easy to check that
\begin{equation}\label{equ 3.1}
\left(\partial_s+\mathcal{L}\right)G_N(s, \eta_s)=\frac{1}{N}G_N(s, \eta_s)-\left(\eta_s(\mathbf{0})-\mathbf{E}_{\nu_{\rho, N}}\eta_s(\mathbf{0})\right).
\end{equation}

A proof of \eqref{equ 3.1} in the homogeneous case is given in \cite{Xue2026}, which is also applicable in the inhomogeneous case and hence here we omit a detailed check of \eqref{equ 3.1}. By \eqref{equ 3.1}, we have
\begin{equation}\label{equ 3.2}
\frac{1}{\sqrt{N}}\xi_{tN}^N=\frac{1}{\sqrt{N}}M_{tN}^N-\left(\frac{1}{\sqrt{N}}\left(G_N(tN, \eta_{tN})-G_N(0, \eta_0)\right)
-\frac{1}{N^{3/2}}\int_0^{tN}G_N(s, \eta_s)ds\right).
\end{equation}

The following lemma shows that the term
\[
\frac{1}{\sqrt{N}}\left(G_N(tN, \eta_{tN})-G_N(0, \eta_0)\right)
-\frac{1}{N^{3/2}}\int_0^{tN}G_N(s, \eta_s)ds
\]
in the decomposition \eqref{equ 3.2} converges weakly to $0$.
\begin{lemma}\label{lemma 3.1}
Let $\eta_0$ be distributed with $\nu_{\rho, N}$. For any $t\geq 0$,
\[
\lim_{N\rightarrow+\infty}\frac{1}{\sqrt{N}}\left(G_N(tN, \eta_{tN})-G_N(0, \eta_0)\right)
-\frac{1}{N^{3/2}}\int_0^{tN}G_N(s, \eta_s)ds=0
\]
in probability.
\end{lemma}

\proof[Proof of Lemma \ref{lemma 3.1}]
For any $x\in \mathbb{Z}^d$, we define
\[
\Phi(x)=\hat{\mathbb{P}}\left(X_t=\mathbf{0}\text{~for some~}t>0\big|X_0=x\right),
\]
where $\{X_t\}_{t\geq 0}$ is the simple random walk defined as in Section \ref{section two}.
According to an argument similar to that given in the proof of Lemma and 5.1 of \cite{Xue2026}, which is the homogeneous version of Lemma \ref{lemma 3.1}, to complete this proof we only need to show that there exists $K_1<+\infty$ such that
\begin{equation}\label{equ 3.3}
{\rm Cov}_{\nu_{\rho, N}}\left(\eta_s(x), \eta_s(y)\right)\leq K_1\Phi(x-y)
\end{equation}
for any $x, y\in \mathbb{Z}^d, s\geq 0$.

To check \eqref{equ 3.3}, for any $x, y\in \mathbb{Z}^d$, we denote by $\{X_t^x\}_{t\geq 0}$ a version of $\{X_t\}_{t\geq 0}$ starting from $x$ and by $\{X_t^y\}_{t\geq 0}$ a version of $\{X_t\}_{t\geq 0}$ starting from $y$. We assume that $\{X_t^x\}_{t\geq 0}$ and $\{X_t^y\}_{t\geq 0}$ are independent. Then we define
\[
\tau_{x, y}=\inf\left\{t\geq 0:~X_t^x=X_t^y\right\}.
\]
We further define $\{\mathcal{X}_t^y\}_{t\geq 0}$ as
\[
\mathcal{X}_t^y=
\begin{cases}
X_t^y & \text{~if~}t\leq \tau_{x, y},\\
X_t^x & \text{~if~}t>\tau_{x, y},
\end{cases}
\]
i.e., $\left\{\left(X_t^x, \mathcal{X}_t^y\right)\right\}_{t\geq 0}$ is the coalescing random walk starting from $(x, y)$. We denote by $\tilde{\mathbb{P}}$ the probability measure of $\{X_t^x\}_{t\geq 0}, \{X_t^y\}_{t\geq 0}, \{\mathcal{X}_t^y\}_{t\geq 0}$ and by $\tilde{\mathbb{E}}$ the expectation with respect to $\tilde{\mathbb{P}}$. By Proposition A.1 of \cite{Xue2026} and the definition of $\nu_{\rho, N}$,
\begin{align*}
\mathbb{E}_{\nu_{\rho, N}}\left(\eta_s(x)\eta_s(y)\right)
&=\tilde{\mathbb{E}}\left(\rho\left(X_s^x/\sqrt{N}\right)\rho\left(\mathcal{X}_s^y/\sqrt{N}\right)1_{\{\tau_{x, y}>s\}}\right)\\
&\text{\quad\quad}+\tilde{\mathbb{E}}\left(\rho\left(X_s^x/\sqrt{N}\right)1_{\{\tau_{x, y}\leq s\}}\right)
\end{align*}
and
\[
\mathbb{E}_{\nu_{\rho, N}}\left(\eta_s(x)\right)\mathbb{E}_{\nu_{\rho, N}}\left(\eta_s(y)\right)
=\tilde{\mathbb{E}}\left(\rho\left(X_s^x/\sqrt{N}\right)\right)\tilde{\mathbb{E}}\left(\rho\left(X_s^y/\sqrt{N}\right)\right).
\]
Since $X_s^x$ and $X_s^y$ are independent,
\begin{align*}
\mathbb{E}_{\nu_{\rho, N}}\left(\eta_s(x)\right)\mathbb{E}_{\nu_{\rho, N}}\left(\eta_s(y)\right)
&=\tilde{\mathbb{E}}\left(\rho\left(X_s^x/\sqrt{N}\right)\rho\left(X_s^y/\sqrt{N}\right)\right)\\
&=\tilde{\mathbb{E}}\left(\rho\left(X_s^x/\sqrt{N}\right)\rho\left(X_s^y/\sqrt{N}\right)1_{\{\tau_{x, y}>s\}}\right)\\
&\text{\quad\quad}+\tilde{\mathbb{E}}\left(\rho\left(X_s^x/\sqrt{N}\right)\rho\left(X_s^y/\sqrt{N}\right)1_{\{\tau_{x, y}\leq s\}}\right)\\
&=\tilde{\mathbb{E}}\left(\rho\left(X_s^x/\sqrt{N}\right)\rho\left(\mathcal{X}_s^y/\sqrt{N}\right)1_{\{\tau_{x, y}>s\}}\right)\\
&\text{\quad\quad}+\tilde{\mathbb{E}}\left(\rho\left(X_s^x/\sqrt{N}\right)\rho\left(X_s^y/\sqrt{N}\right)1_{\{\tau_{x, y}\leq s\}}\right).
\end{align*}
Therefore,
\begin{align*}
{\rm Cov}_{\nu_{\rho, N}}\left(\eta_s(x), \eta_s(y)\right)&\leq \tilde{\mathbb{E}}\left(\rho\left(X_s^x/\sqrt{N}\right)1_{\{\tau_{x, y}\leq s\}}\right)\\
&\text{\quad\quad}-\tilde{\mathbb{E}}\left(\rho\left(X_s^x/\sqrt{N}\right)\rho\left(X_s^y/\sqrt{N}\right)1_{\{\tau_{x, y}\leq s\}}\right)\\
&\leq \tilde{\mathbb{P}}\left(\tau_{x, y}\leq s\right)\leq \tilde{\mathbb{P}}\left(\tau_{x, y}<+\infty\right).
\end{align*}
Since $\{X_t^x-X_t^y\}_{t\geq 0}$ is a copy of $\{X_{2t}\}_{t\geq 0}$ starting from $x-y$, we have
\[
\tilde{\mathbb{P}}\left(\tau_{x, y}<+\infty\right)=\Phi(x-y)
\]
and then \eqref{equ 3.3} holds with $K_1=1$. Since \eqref{equ 3.3} holds, the proof is complete.
\qed

Now we deal with the martingale term in the decomposition \eqref{equ 3.2}.

\begin{lemma}\label{lemma 3.2}
Let $\eta_0$ be distributed with $\nu_{\rho, N}$. For any $T>0$, $\left\{\frac{1}{\sqrt{N}}M_{tN}^N:~0\leq t\leq T\right\}$ converges weakly, with respect to the Skorohod topology, to $\left\{\int_0^tA_{s, d}dB_s:~0\leq t\leq T\right\}$ as $N\rightarrow+\infty$.
\end{lemma}

\proof[Proof of Lemma \ref{lemma 3.2}]
We denote by $\{\langle M^N\rangle_t\}_{t\geq 0}$ the quadratic variation process of $\{M^N_t\}_{t\geq 0}$. According to an argument similar to that given in the proof of  Lemma 5.2 of \cite{Xue2026}, which is the homogeneous version of Lemma \ref{lemma 3.2}, to complete this proof we only need to show that
\begin{equation}\label{equ 3.4}
\lim_{N\rightarrow+\infty}\frac{1}{N}\mathbb{E}_{\nu_{\rho, N}}\langle M^N\rangle_{tN}=\int_0^t A^2_{s, d}ds
\end{equation}
for any $t\geq 0$ and
\begin{equation}\label{equ 3.5}
\lim_{r\rightarrow+\infty}\sup_{x\sim y, z\sim w, s\geq 0, N\geq 1}{\rm Cov}_{\nu_{\rho, N}}\left(\left(\eta_{s}(y)-\eta_{s}(x)\right)^2, \left(\eta_{s+r}(z)-\eta_{s+r}(w)\right)^2\right)=0.
\end{equation}
The proof of \eqref{equ 3.5} in the case $\rho\equiv p\in (0, 1)$, which was given in Section 3 of \cite{Xue2026}, only utilizes the fact that the initial distribution of the voter model is a product measure, i.e., does not rely on the homogeneity assumption and hence is also applicable in the inhomogeneous case. Therefore, we only need to check \eqref{equ 3.4} to complete this proof. According to Dynkin's martingale formula,
\[
\langle M^N\rangle_t=\int_0^t\sum_{x\in \mathbb{Z}^d}\sum_{y: y\sim x}g_N^2(x)\left(\eta_s(y)-\eta_s(x)\right)^2ds
\]
and hence
\[
\frac{1}{N}\langle M^N\rangle_{tN}=\int_0^t\sum_{x\in \mathbb{Z}^d}\sum_{y: y\sim x}g_N^2(x)\left(\eta_{sN}(y)-\eta_{sN}(x)\right)^2ds.
\]
By Lemma 5.3 of \cite{Xue2026},
\begin{equation}\label{equ 3.6}
\lim_{N\rightarrow+\infty}\sum_{x\in \mathbb{Z}^d}g_N^2(x)=\sum_{x\in \mathbb{Z}^d}g_\infty^2(x)=\int_0^{+\infty}\theta p_\theta(\mathbf{0}, \mathbf{0})d\theta,
\end{equation}
where
\[
g_\infty(x)=\int_0^{+\infty}p_s(\mathbf{0}, x)ds.
\]
For any $x\in \mathbb{Z}^d$, we denote by $\|x\|_1$ the $l_1$-norm of $x$. According to the definition of $g_N$, we have
\[
\sum_{x: \|x\|_1\geq K}g_N^2(x)\leq \sum_{x:\|x\|_1\geq K}g_\infty^2(x)
\]
for any $K>0$ and hence
\begin{equation}\label{equ 3.7}
\lim_{K\rightarrow+\infty}\sup_{N\geq 1}\sum_{x: \|x\|_1\geq K}g_N^2(x)\leq \lim_{K\rightarrow+\infty}\sum_{x: \|x\|_1\geq K}g_\infty^2(x)=0.
\end{equation}
By \eqref{equ 3.6}, \eqref{equ 3.7} and the definition of $A_{s, d}$, to check \eqref{equ 3.4} we only need to show that
\begin{equation}\label{equ 3.8}
\lim_{N\rightarrow+\infty}\sup_{x, y: \|x\|_1\leq K, y\sim x}\left|\mathbb{E}_{\nu_{\rho, N}}\left(\left(\eta_{sN}(y)-\eta_{sN}(x)\right)^2\right)
-2\gamma_d\varrho(s, \mathbf{0})\left(1-\varrho(s, \mathbf{0})\right)\right|=0
\end{equation}
for any $s>0$ and $K>0$.

Now we check \eqref{equ 3.8}. For $x\sim y$, let $X_t^x, X_t^y, \mathcal{X}_t^y, \tau_{x, y}$ be defined as in the proof of Lemma \ref{lemma 3.1}. According to Proposition A. 1 of \cite{Xue2026},
\[
\mathbb{E}_{\nu_{\rho, N}}\left(\left(\eta_{sN}(y)-\eta_{sN}(x)\right)^2\right)=\tilde{\mathbb{E}}\mathcal{H}\left(\mathcal{X}_{sN}^y, X_{sN}^x\right),
\]
where
\[
\mathcal{H}(z, w)=\mathbb{E}_{\nu_{\rho, N}}\left(\left(\eta_0(z)-\eta_0(w)\right)^2\right)
\]
for any $z, w\in \mathbb{Z}^d$. On the event $\tau_{x, y}\leq sN$, $\mathcal{X}_{sN}^y=X_{sN}^x$ and hence $\mathcal{H}\left(\mathcal{X}_{sN}^y, X_{sN}^x\right)=0$. For $z\neq w$, according to the definition of $\nu_{\rho, N}$,
\[
\mathcal{H}(z, w)=\rho(z/\sqrt{N})(1-\rho(w/\sqrt{N}))+\rho(w/\sqrt{N})(1-\rho(z/\sqrt{N})).
\]
Hence, via the fact that $\mathcal{X}_{sN}^y=X_{sN}^y$ when $\tau_{x, y}>sN$,
\begin{align}\label{equ 3.9}
&\mathbb{E}_{\nu_{\rho, N}}\left(\left(\eta_{sN}(y)-\eta_{sN}(x)\right)^2\right) \\
&=\tilde{\mathbb{E}}\left(1_{\{\tau_{x, y}>sN\}}\left(\rho(X_{sN}^x/\sqrt{N})(1-\rho(X_{sN}^y/\sqrt{N}))+\rho(X_{sN}^y/\sqrt{N})(1-\rho(X_{sN}^x/\sqrt{N}))\right)\right), \notag
\end{align}
where $1_A$ is the indicator function of the event $A$. Since $X_{sN}^x$ and $X_{sN}^y$ are independent,
\begin{equation}\label{equ 3.10}
\tilde{\mathbb{E}}\left(\rho(X_{sN}^x/\sqrt{N})(1-\rho(X_{sN}^y/\sqrt{N}))\right)
=\tilde{\mathbb{E}}\left(\rho(X_{sN}^x/\sqrt{N})\right)\tilde{\mathbb{E}}\left(1-\rho(X_{sN}^y/\sqrt{N})\right).
\end{equation}
Since $X_{sN}^x$ (resp. $X_{sN}^y$) and $X_{sN}^\mathbf{0}+x$ (resp. $X_{sN}^\mathbf{0}+y$) have the same probability distribution, for any $x, y\in \mathbb{Z}^d$ such that $x\sim y$ and $\|x\|_1\leq K$, we have
\begin{equation}\label{equ 3.11}
\left|\tilde{\mathbb{E}}\left(\rho(X_{sN}^x/\sqrt{N})\right)\tilde{\mathbb{E}}\left(1-\rho(X_{sN}^y/\sqrt{N})\right)-\varrho(s, \mathbf{0})(1-\varrho(s, \mathbf{0}))\right|\leq 2(\varepsilon_1^N+\varepsilon_2^N),
\end{equation}
where
\[
\varepsilon_1^N=\left|\tilde{\mathbb{E}}\left(\rho(X_{sN}^\mathbf{0}/\sqrt{N})\right)-\varrho(s, \mathbf{0})\right|
=\left|\tilde{\mathbb{E}}\left(\rho(X_{sN}^\mathbf{0}/\sqrt{N})\right)-\mathbb{E}\rho(\mathcal{W}^d_{2s})\right|
\]
and
\[
\varepsilon_2^N=\sup\{\left|\rho(u)-\rho(v)\right|:~u, v\in \mathbb{R}^d, \|u-v\|_1\leq (K+1)/\sqrt{N}\}.
\]
By Donsker's invariance principle and the uniform continuity of $\rho$,
\[
\lim_{N\rightarrow+\infty}\varepsilon_1^N=\lim_{N\rightarrow+\infty}\varepsilon_2^N=0.
\]
Therefore, by \eqref{equ 3.10} and \eqref{equ 3.11},
\begin{equation}\label{equ 3.12}
\lim_{N\rightarrow+\infty}\sup_{x, y: \|x\|_1\leq K, y\sim x}\left|\tilde{\mathbb{E}}\left(\rho(X_{sN}^x/\sqrt{N})(1-\rho(X_{sN}^y/\sqrt{N}))\right)
-\varrho(s, \mathbf{0})(1-\varrho(s, \mathbf{0}))\right|=0.
\end{equation}
According to an argument similar to that leading to \eqref{equ 3.12}, we have
\begin{equation}\label{equ 3.13}
\lim_{N\rightarrow+\infty}\sup_{x, y: \|x\|_1\leq K, y\sim x}\left|\tilde{\mathbb{E}}\left(\rho(X_{sN}^y/\sqrt{N})(1-\rho(X_{sN}^x/\sqrt{N}))\right)
-\varrho(s, \mathbf{0})(1-\varrho(s, \mathbf{0}))\right|=0.
\end{equation}
By \eqref{equ 3.9}, \eqref{equ 3.12} and \eqref{equ 3.13}, to check \eqref{equ 3.8} we only need to show that
\begin{align}\label{equ 3.14}
\lim_{N\rightarrow+\infty}&\sup_{x, y: \|x\|_1\leq K, y\sim x}\Big|\tilde{\mathbb{E}}\left(1_{\{\tau_{x, y}\leq sN\}}\rho(X_{sN}^x/\sqrt{N})(1-\rho(X_{sN}^y/\sqrt{N}))\right)\notag\\
&\text{\quad\quad\quad\quad}-(1-\gamma_d)\varrho(s, \mathbf{0})(1-\varrho(s, \mathbf{0}))\Big|=0
\end{align}
and
\begin{align}\label{equ 3.15}
\lim_{N\rightarrow+\infty}&\sup_{x, y: \|x\|_1\leq K, y\sim x}\Big|\tilde{\mathbb{E}}\left(1_{\{\tau_{x, y}\leq sN\}}\rho(X_{sN}^y/\sqrt{N})(1-\rho(X_{sN}^x/\sqrt{N}))\right)\notag\\
&\text{\quad\quad\quad\quad}-(1-\gamma_d)\varrho(s, \mathbf{0})(1-\varrho(s, \mathbf{0}))\Big|=0.
\end{align}
Equations \eqref{equ 3.14} and \eqref{equ 3.15} follow from the same argument and hence here we only check \eqref{equ 3.14}. For simplicity, we denote the term
\[
\sup_{x, y: \|x\|_1\leq K, y\sim x}\Big|\tilde{\mathbb{E}}\left(1_{\{\tau_{x, y}\leq sN\}}\rho(X_{sN}^x/\sqrt{N})(1-\rho(X_{sN}^y/\sqrt{N}))\right)-(1-\gamma_d)\varrho(s, \mathbf{0})(1-\varrho(s, \mathbf{0}))\Big|
\]
by ${\rm \uppercase\expandafter{\romannumeral1}}_N$. We denote $(1, 0, \ldots, 0)$ by $e_1$. For any $K_2>0$ and $x\sim y$,
\[
\tilde{\mathbb{P}}(K_2<\tau_{x, y}\leq sN)\leq \tilde{\mathbb{P}}(K_2<\tau_{x,y}<+\infty)=\tilde{\mathbb{P}}(K_2<\tau_{\mathbf{0},e_1}<+\infty)
\]
and
\[
(1-\gamma_d)-\tilde{\mathbb{P}}(\tau_{0, e_1}\leq K_2)=\tilde{\mathbb{P}}(K_2<\tau_{\mathbf{0},e_1}<+\infty).
\]
Hence,
\[
\left|{\rm \uppercase\expandafter{\romannumeral1}}_N-{\rm \uppercase\expandafter{\romannumeral2}}_N^{K_2}\right|\leq 2\varepsilon_3(K_2),
\]
where
\[
\varepsilon_3(K_2)=\tilde{\mathbb{P}}(K_2<\tau_{\mathbf{0},e_1}<+\infty)
\]
and
\begin{align*}
{\rm \uppercase\expandafter{\romannumeral2}}_N^{K_2}=&\sup_{x, y: \|x\|_1\leq K, y\sim x}\Big|\tilde{\mathbb{E}}\left(1_{\{\tau_{x, y}\leq K_2\}}\rho(X_{sN}^x/\sqrt{N})(1-\rho(X_{sN}^y/\sqrt{N}))\right)\\
&\text{\quad\quad\quad\quad}-\tilde{\mathbb{P}}(\tau_{\mathbf{0}, e_1}\leq K_2)\varrho(s, \mathbf{0})(1-\varrho(s, \mathbf{0}))\Big|.
\end{align*}
Since $\lim_{K_2\rightarrow+\infty}\varepsilon_3(K_2)=\tilde{\mathbb{P}}(\emptyset)=0$, to check \eqref{equ 3.14} we only need to show that
\begin{equation}\label{equ 3.16}
\lim_{N\rightarrow+\infty}{\rm \uppercase\expandafter{\romannumeral2}}_N^{K_2}=0
\end{equation}
for any $K_2>0$. For any $K_3>0$ and $x\sim y$, we denote by $A_{x, y}(K_3)$ the event that
\[
\|X_{\tau_{x, y}}^x-x\|_1>K_3.
\]
On the event $A_{x, y}(K_3)\bigcap \{\tau_{x, y}\leq K_2\}$, we have
\[
\sup_{0\leq t\leq K_2}\|X_t^x-x\|_1>K_3
\]
and hence
\[
\tilde{\mathbb{P}}\left(A_{x, y}(K_3)\bigcap \{\tau_{x, y}\leq K_2\}\right)
\leq \tilde{\mathbb{P}}\left(\sup_{0\leq t\leq K_2}\|X_t^\mathbf{0}\|_1>K_3\right).
\]
Since
\[
\lim_{K_3\rightarrow+\infty}\tilde{\mathbb{P}}\left(\sup_{0\leq t\leq K_2}\|X_t^\mathbf{0}\|_1>K_3\right)=0,
\]
to check \eqref{equ 3.16} we only need to show that
\begin{equation}\label{equ 3.17}
\lim_{N\rightarrow+\infty}{\rm \uppercase\expandafter{\romannumeral3}}_N^{K_2, K_3}=0
\end{equation}
for any $K_2, K_3>0$, where
\begin{align*}
{\rm \uppercase\expandafter{\romannumeral3}}_N^{K_2, K_3}=&\sup_{x, y: \|x\|_1\leq K, y\sim x}\Big|\tilde{\mathbb{E}}\left(1_{\{\tau_{x, y}\leq K_2\}}1_{A_{x, y}^c(K_3)}\rho(X_{sN}^x/\sqrt{N})(1-\rho(X_{sN}^y/\sqrt{N}))\right)\\
&\text{\quad\quad\quad\quad}-\tilde{\mathbb{P}}\left(\tau_{\mathbf{0}, e_1}\leq K_2, A^c_{\mathbf{0}, e_1}(K_3)\right)\varrho(s, \mathbf{0})(1-\varrho(s, \mathbf{0}))\Big|
\end{align*}
and $A^c$ is the of complement of $A$. According to the strong Markov property of the simple random walk,
\[
\tilde{\mathbb{E}}\left(\rho(X_{sN}^x/\sqrt{N})(1-\rho(X_{sN}^y/\sqrt{N}))\Big|\tau_{x, y}, X_{\tau_{x, y}}^x\right)
=\mathcal{H}_2(X_{\tau_{x, y}}^x, sN-\tau_{x, y}),
\]
where
\[
\mathcal{H}_2(z, r)=\tilde{\mathbb{E}}\rho(X_r^z/\sqrt{N})\tilde{\mathbb{E}}\left(1-\rho(X_r^z/\sqrt{N})\right)
\]
for any $z\in \mathbb{Z}^d$ and $r>0$. By the law of iterated expectations,
\begin{align*}
&\tilde{\mathbb{E}}\left(1_{\{\tau_{x, y}\leq K_2\}}1_{A_{x, y}^c(K_3)}\rho(X_{sN}^x/\sqrt{N})(1-\rho(X_{sN}^y/\sqrt{N}))\right)\\
&=\tilde{\mathbb{E}}\left(1_{\{\tau_{x, y}\leq K_2\}}1_{A_{x, y}^c(K_3)}\mathcal{H}_2(X_{\tau_{x, y}}^x, sN-\tau_{x, y})\right).
\end{align*}
Therefore,
\begin{align}\label{equ 3.18}
{\rm \uppercase\expandafter{\romannumeral3}}_N^{K_2, K_3}=&\sup_{x, y: \|x\|_1\leq K, y\sim x}\Big|\tilde{\mathbb{E}}\left(1_{\{\tau_{x, y}\leq K_2\}}1_{A_{x, y}^c(K_3)}\rho(X_{sN}^x/\sqrt{N})(1-\rho(X_{sN}^y/\sqrt{N}))\right)\notag\\
&\text{\quad\quad\quad\quad}-\tilde{\mathbb{P}}\left(\tau_{x, y}\leq K_2, A^c_{x, y}(K_3)\right)\varrho(s, \mathbf{0})(1-\varrho(s, \mathbf{0}))\Big|\notag\\
&\leq {\rm \uppercase\expandafter{\romannumeral4}}_N^{K_2, K_3},
\end{align}
where
\[
{\rm \uppercase\expandafter{\romannumeral4}}_N^{K_2, K_3}=\sup_{z, t: \|z\|_1\leq (K+K_3), t\leq K_2}\left|\mathcal{H}_2(z, sN-t)-\varrho(s, \mathbf{0})(1-\varrho(s, \mathbf{0}))\right|.
\]
Since $\varrho(s, \mathbf{0})=\mathbb{E}\rho(\mathcal{W}_{2s}^d)$, according to the Donsker's invariance principle of the simple random walk,
\[
\lim_{N\rightarrow+\infty}\sup_{z, t: \|z\|_1\leq (K+K_3), t\leq K_2}|\tilde{\mathbb{E}}\rho(X_{sN-t}^z/\sqrt{N})-\varrho(s, \mathbf{0})|=0
\]
and hence
\begin{equation}\label{equ 3.19}
\lim_{N\rightarrow+\infty}{\rm \uppercase\expandafter{\romannumeral4}}_N^{K_2, K_3}=0.
\end{equation}
Equation \eqref{equ 3.17} follows from \eqref{equ 3.18} and \eqref{equ 3.19}. Since \eqref{equ 3.17} holds, the proof is complete.
\qed

At last we prove Theorem \ref{theorem 2.1 main d geq 4} in the case $d\geq 5$.

\proof[Proof of Theorem \ref{theorem 2.1 main d geq 4} in the case $d\geq 5$]
By Lemmas \ref{lemma 3.1}, \ref{lemma 3.2} and Equation \eqref{equ 3.2}, we only need to show that
\[
\left\{\frac{1}{\sqrt{N}}\xi_{tN}^{N, d}:~0\leq t\leq T\right\}_{N\geq 1}
\]
are tight under the uniform topology of $C[0, T]$. In the homogeneous case where $\rho\equiv p\in (0, 1)$, it is shown in Section 5 of \cite{Xue2026} that there exists $K_4<+\infty$ independent of $t, s\geq 0$ and $N\geq 1$ such that
\begin{equation}\label{equ 3.20}
\mathbb{E}_{\nu_{\rho, N}}\left(\left(\frac{1}{\sqrt{N}}\xi_{tN}^{N, d}-\frac{1}{\sqrt{N}}\xi_{sN}^{N, d}\right)^4\right)\leq K_4(t-s)^2
\end{equation}
for any $0\leq s<t$ and $N\geq 1$. The proof of \eqref{equ 3.20} given in \cite{Xue2026} only utilizes the fact that the initial distribution is a product measure, i.e, does not rely on the homogeneity assumption and is also applicable in the inhomogeneous case. According to Corollary 14.9 of \cite{Kallenberg1997}, the tightness of
\[
\left\{\frac{1}{\sqrt{N}}\xi_{tN}^{N, d}:~0\leq t\leq T\right\}_{N\geq 1}
\]
follows from \eqref{equ 3.20}. Therefore, the proof is complete.
\qed

\subsection{Proof of Theorem \ref{theorem 2.1 main d geq 4}: The case $d=4$}\label{subsection 3.2}
In this subsection, we prove Theorem \ref{theorem 2.1 main d geq 4} in the case $d=4$. So throughout this subsection we assume that $d=4$. Let $g_N(x)$, $G_N(t, \eta)$ and $M_t^N$ be defined as in Subsection \ref{subsection 3.1} except that all $\mathbb{Z}^d$ in the expressions of these notations are replaced by $\mathbb{Z}^4$. Then we have the following analogue of \eqref{equ 3.2},
\begin{align}\label{equ 3.21}
&\frac{1}{\sqrt{N\log N}}\xi_{tN}^N\\
&=\frac{M_{tN}^N}{\sqrt{N\log N}}-\left(\frac{G_N(tN, \eta_{tN})-G_N(0, \eta_0)}{\sqrt{N\log N}}
-\frac{1}{N^{3/2}\sqrt{\log N}}\int_0^{tN}G_N(s, \eta_s)ds\right). \notag
\end{align}
According to an argument similar to that given in the proof of Lemma \ref{lemma 3.1}, we have a $\mathbb{Z}^4$-version of \eqref{equ 3.3}
and the following lemma.
\begin{lemma}\label{lemma 3.3}
Let $\eta_0$ be distributed with $\nu_{\rho, N}$. For any $t\geq 0$,
\[
\lim_{N\rightarrow+\infty}\frac{1}{\sqrt{N\log N}}\left(G_N(tN, \eta_{tN})-G_N(0, \eta_0)\right)
-\frac{1}{N^{3/2}\sqrt{\log N}}\int_0^{tN}G_N(s, \eta_s)ds=0
\]
in probability.
\end{lemma}
Now we deal with the term $\frac{M_{tN}^N}{\sqrt{N\log N}}$ in \eqref{equ 3.21}. We have the following analogue of Lemma \ref{lemma 3.2}.
\begin{lemma}\label{lemma 3.4}
Let $\eta_0$ be distributed with $\nu_{\rho, N}$. For any $T>0$, $\left\{\frac{1}{\sqrt{N\log N}}M_{tN}^N:~0\leq t\leq T\right\}$ converges weakly, with respect to the Skorohod topology, to $\left\{\int_0^tA_{s, 4}dB_s:~0\leq t\leq T\right\}$ as $N\rightarrow+\infty$.
\end{lemma}

\proof[Proof of Lemma \ref{lemma 3.4}]
According to an argument similar to that leading to Lemma \ref{lemma 3.2}, we only need to show the following analogue of \eqref{equ 3.4},
\begin{equation}\label{equ 3.22}
\lim_{N\rightarrow+\infty}\frac{1}{N\log N}\mathbb{E}_{\nu_{\rho, N}}\langle M^N\rangle_{tN}=\int_0^t A^2_{s, 4}ds.
\end{equation}
Similar to that in the proof of Lemma \ref{lemma 3.2},
\begin{equation}\label{equ 3.23}
\frac{1}{N\log N}\langle M^N\rangle_{tN}=\int_0^t\frac{1}{\log N}\sum_{x\in \mathbb{Z}^d}\sum_{y: y\sim x}g_N^2(x)\left(\eta_{sN}(y)-\eta_{sN}(x)\right)^2ds.
\end{equation}
By Lemma 4.3 of \cite{Xue2026}, we have
\begin{equation}\label{equ 3.24}
\lim_{N\rightarrow+\infty}\frac{1}{\log N}\sum_{x\in \mathbb{Z}^d}g_N^2(x)=\frac{1}{16\pi^2}.
\end{equation}
For any $\epsilon>0, s>0$, since $\lim_{N\rightarrow+\infty}\frac{N^{\frac{1}{2}-\epsilon}}{\sqrt{N}}=0$, we have
\begin{equation}\label{equ 3.25}
\lim_{N\rightarrow+\infty}\sup_{x, y: \|x\|_1\leq N^{\frac{1}{2}-\epsilon}, y\sim x}\left|\mathbb{E}_{\nu_{\rho, N}}\left(\left(\eta_{sN}(y)-\eta_{sN}(x)\right)^2\right)
-2\gamma_4\varrho(s, \mathbf{0})\left(1-\varrho(s, \mathbf{0})\right)\right|=0
\end{equation}
according to an argument similar to that leading to \eqref{equ 3.8}. By \eqref{equ 3.23}-\eqref{equ 3.25} and the fact that each $x\in \mathbb{Z}^4$ has $8$ neighbors, to check \eqref{equ 3.22} we only need to show that
\begin{equation}\label{equ 3.26}
\lim_{\epsilon\rightarrow 0}\limsup_{N\rightarrow+\infty}\frac{1}{\log N}\sum_{x:\|x\|_1>N^{\frac{1}{2}-\epsilon}}g_N^2(x)=0.
\end{equation}
According to the definition of $g_N(x)$ and the Markov property of $\{X_t\}_{t\geq 0}$,
\begin{align*}
\sum_{x:\|x\|_1>N^{\frac{1}{2}-\epsilon}}g_N^2(x)
&=\int_0^{+\infty}\int_0^{+\infty}e^{-(s_1+s_2)/N}\tilde{\mathbb{P}}\left(X_{s_1+s_2}^\mathbf{0}=\mathbf{0}, \|X_{s_1}^\mathbf{0}\|_1>N^{\frac{1}{2}-\epsilon}\right)ds_1ds_2\\
&=N^2\int_0^{+\infty}e^{-u}\left(\int_0^u\tilde{\mathbb{P}}\left(X_{Nu}^\mathbf{0}=\mathbf{0}, \|X_{Nr}^\mathbf{0}\|_1>N^{\frac{1}{2}-\epsilon}\right)dr\right)du.
\end{align*}
According to the moderate deviation principle of the simple random walk,
\[
\sup_{0\leq r\leq N^{-4\epsilon}}\tilde{\mathbb{P}}\left(\|X_{Nr}^\mathbf{0}\|_1>N^{\frac{1}{2}-\epsilon}\right)=\exp\{-\Theta(N^{2\epsilon})\}.
\]
Hence,
\[
\sum_{x:\|x\|_1>N^{\frac{1}{2}-\epsilon}}g_N^2(x)\leq {\rm \uppercase\expandafter{\romannumeral5}}_N(\epsilon)+o(1),
\]
where
\[
{\rm \uppercase\expandafter{\romannumeral5}}_N(\epsilon)=N^2\int_{N^{-4\epsilon}}^{+\infty}ue^{-u}p_{Nu, 4}(\mathbf{0}, \mathbf{0})du.
\]
According to the local central limit theorem of the simple random walk, there exists $K_5<+\infty$ independent of $t$ such that
\[
t^2p_{t, 4}(\mathbf{0}, \mathbf{0})\leq K_5
\]
for all $t\geq 0$. Hence,
\begin{align*}
{\rm \uppercase\expandafter{\romannumeral5}}_N(\epsilon)&\leq
K_5N^2\int^{+\infty}_{N^{-4\epsilon}}ue^{-u}\frac{1}{N^2u^2}du\\
&=K_5\int_{N^{-4\epsilon}}^{+\infty}\frac{1}{u}e^{-u}du\\
&\leq O(1)+K_5\int_{N^{-4\epsilon}}^1\frac{1}{u}du=O(1)+4K_5\epsilon \log N.
\end{align*}
As a result,
\[
\limsup_{N\rightarrow+\infty}\frac{1}{\log N}\sum_{x:\|x\|_1>N^{\frac{1}{2}-\epsilon}}g_N^2(x)\leq 4\epsilon K_5
\]
and hence \eqref{equ 3.26} holds. Therefore, the proof is complete.
\qed

At last, we prove Theorem \ref{theorem 2.1 main d geq 4} in the case $d=4$.

\proof[Proof of Theorem \ref{theorem 2.1 main d geq 4} in the case $d=4$]
It is shown in Section 4 of \cite{Xue2026} that, when $\rho\equiv p\in (0, 1)$, there exists $K_6<+\infty$ independent of $t,s>0$ and $N\geq 1$ such that
\begin{equation}\label{equ 3.27}
\mathbb{E}_{\nu_{\rho, N}}\left(\left(\frac{1}{\sqrt{N\log N}}\xi_{tN}^{N, 4}-\frac{1}{\sqrt{N\log N}}\xi_{sN}^{N, 4}\right)^4\right)\leq K_6(t-s)^2
\end{equation}
for any $0\leq s<t$ and $N\geq 1$. The proof of \eqref{equ 3.27} given in \cite{Xue2026} does not rely on the homogeneous assumption and hence \eqref{equ 3.27} holds for general $\rho$. By \eqref{equ 3.27} and Corollary 14.9 of \cite{Kallenberg1997},
\[
\left\{\frac{1}{\sqrt{N\log N}}\xi_{tN}^{N, 4}:~0\leq t\leq T\right\}_{N\geq 1}
\]
are tight under the uniform topology of $C[0, T]$. Theorem \ref{theorem 2.1 main d geq 4} in the case $d=4$ follows from Lemmas \ref{lemma 3.3}, \ref{lemma 3.4}, Equation \eqref{equ 3.21} and the above tightness.
\qed

\section{Proof of Theorem \ref{theorem 2.2 main d=3}}\label{section four}
In this section, we prove Theorem \ref{theorem 2.2 main d=3}. So throughout this section we assume that $d=3$. We still only highlight the main new difficulties generated from the inhomogeneous assumption. For any $x\in \mathbb{Z}^3$, $0\leq s\leq t, N\geq 1$ and $\eta\in \{0, 1\}^{\mathbb{Z}^d}$, we define
\[
v(t, x)=\int_0^tp_s(\mathbf{0}, x)ds
\]
and
\[
V_s^{t, N}(\eta)=\sum_{x\in \mathbb{Z}^3}\left(\eta(x)-\mathbb{E}_{\nu_{\rho, N}}\eta_s(x)\right)v(t-s, x).
\]
For $0\leq s\leq t$, we further define
\[
\mathcal{M}_s^{t, N}=V_s^{t, N}(\eta_s)-V_0^{t, N}(\eta_0)-\int_0^s(\partial_r+\mathcal{L})V_r^{t, N}(\eta_r)dr.
\]
Then, $\{\mathcal{M}_s^{t, N}\}_{s\leq t}$ is a martingale. By Kolmogorov-Chapman equation,
\[
\partial_sp_s(\mathbf{0}, x)=\sum_{y:y\sim x}\left(p_s(\mathbf{0}, y)-p_s(\mathbf{0}, x)\right)
\]
and
\[
\partial_s\mathbb{E}_{\nu_{\rho, N}}\eta_s(x)=\sum_{y:y\sim x}\left(\partial_s\mathbb{E}_{\nu_{\rho, N}}\eta_s(y)-\partial_s\mathbb{E}_{\nu_{\rho, N}}\eta_s(x)\right).
\]
Then, according an argument similar to that leading to (3.7) of \cite{Xue2026}, we have the following martingale decomposition formula,
\begin{equation}\label{equ 4.1}
\frac{1}{N^{\frac{3}{4}}}\xi_{tN}^{N, 3}=\frac{1}{N^{\frac{3}{4}}}\mathcal{M}_{tN}^{tN, N}+\frac{1}{N^{\frac{3}{4}}}V_0^{tN, N}(\eta_0).
\end{equation}
We first deal with the term $\frac{1}{N^{\frac{3}{4}}}V_0^{tN, N}(\eta_0)$ in \eqref{equ 4.1}. We have the following analogue of Lemma 3.2 of \cite{Xue2026}.
\begin{lemma}\label{lemma 4.1}
Let $\eta_0$ be distributed with $\nu_{\rho, N}$. For any $t\geq 0$,
\[
\lim_{N\rightarrow+\infty}\frac{1}{N^{\frac{3}{4}}}V_0^{tN, N}(\eta_0)=0
\]
in probability.
\end{lemma}
According to the fact that ${\rm Var}_{\nu_{\rho, N}}(\eta_0(x))\leq 1$, the argument leading to Lemma 3.2 of \cite{Xue2026} is applicable in the inhomogeneous case and hence we omit a detailed proof of Lemma \ref{lemma 4.1}.

Now we deal with the martingale term in \eqref{equ 4.1}. We have the following lemma.
\begin{lemma}\label{lemma 4.2}
Let $\eta_0$ be distributed with $\nu_{\rho, N}$. For any $T>0$, integer $m\geq 1$ and $0<t_1<t_2<\ldots<t_m\leq T$, 
\[
\frac{1}{N^{\frac{3}{4}}}\left(\mathcal{M}_{t_1N}^{t_1N, N}, \mathcal{M}_{t_2N}^{t_2N, N}, \ldots, \mathcal{M}_{t_mN}^{t_mN, N}\right)
\]
converges weakly to $\sqrt{12\gamma_3}\left(\zeta_{t_1}, \ldots, \zeta_{t_m}\right)$ as $N\rightarrow+\infty$.
\end{lemma}

As a preliminary of the proof of Lemma \ref{lemma 4.2}, we recall some notations and definitions introduced in \cite{Xue2026}.
For $x\sim y$, we denote by $\{N_t^{x, y}\}_{t\geq 0}$ the Poisson process with rate $1$ such that $x$ adopts the opinion of $y$ at each event moment of $N_\cdot^{x, y}$. We further define
\[
\hat{N}_t^{x, y}=N_t^{x, y}-t.
\]
According to an argument similar to that leading to (3.5) of \cite{Xue2026}, we have
\begin{equation}\label{equ 4.2}
\mathcal{M}_s^{t, N}=\sum_{x\in \mathbb{Z}^3}\sum_{y:y\sim x}\int_0^sv(t-r, x)\left(\eta_{r-}(y)-\eta_{r-}(x)\right)d\hat{N}_r^{x, y}
\end{equation}
for all $0\leq s\leq t$.

For all $N\geq 1$, let
\[
\mathbb{Z}^3/\sqrt{N}=\{x/\sqrt{N}:~x\in \mathbb{Z}^3\}.
\]
For any $u\in \mathbb{R}^3$ and $N\geq 1$, we denote by $u_N$ the element in $\mathbb{Z}^3/\sqrt{N}$ such that
\[
u-u_N\in \left(-\frac{1}{2\sqrt{N}}, \frac{1}{2\sqrt{N}}\right]^3.
\]

For each $N\geq 1$, we denote by $\mathcal{Y}^N$ the random field on $[0, T]\times \mathbb{R}^3$ such that
\begin{align*}
\mathcal{Y}^N(H)
&=N^{\frac{1}{4}}\sum_{x\in \mathbb{Z}^3}\sum_{y:y\sim x}\int_0^T\left(\int_{\frac{x}{\sqrt{N}}+(-\frac{1}{2\sqrt{N}}, \frac{1}{2\sqrt{N}}]^3} H(s,u)du\right)\left(\eta_{Ns-}(y)-\eta_{Ns-}(x)\right)d\hat{N}_{Ns}^{x,y}
\end{align*}
for any $H\in C_c\left([0, T]\times \mathbb{R}^3\right)$. We denote by $\mathcal{Y}$ the Gaussian white noise on $[0, T]\times \mathbb{R}^3$ such that
$\mathcal{Y}(H)$ follows the normal distribution with mean zero and variance $\sigma^2(H)$ given by
\[
\sigma^2(H)=12\gamma_3\int_0^T\int_{\mathbb{R}^3}H^2(s,u)\varrho(s, u)\left(1-\varrho(s, u)\right)dsdu
\]
for any  $H\in C_c\left([0, T]\times \mathbb{R}^3\right)$.

The following lemma, which is an analogue of Lemma 3.4 of \cite{Xue2026}, is crucial for us to prove Lemma \ref{lemma 4.2}.
\begin{lemma}\label{lemma 4.3}
Let $\eta_0$ be distributed with $\nu_{\rho, N}$. For any $H\in C_c([0, T]\times \mathbb{R}^3)$, $\mathcal{Y}^N(H)$ converges weakly to $\mathcal{Y}(H)$ as $N\rightarrow+\infty$. 
\end{lemma}

\proof[Proof of Lemma \ref{lemma 4.3}]
According to an argument similar to that leading to Lemma 3.4 of \cite{Xue2026}, to complete this proof we only need to show that
\begin{align}\label{equ 4.3}
&\lim_{N\rightarrow+\infty}N^{\frac{1}{2}}\sum_{x\in \mathbb{Z}^3}\sum_{y:y\sim x}\int_0^t\left(\int_{\frac{x}{\sqrt{N}}+(-\frac{1}{2\sqrt{N}}, \frac{1}{2\sqrt{N}}]^3}h(u)du\right)^2\mathbb{E}_{\nu_{\rho, N}}\left(\left(\eta_{Ns}(y)-\eta_{Ns}(x)\right)^2\right) Nds \notag\\
&=12\gamma_3\int_0^t\left(\int_{\mathbb{R}^3}h^2(u)\varrho(s, u)\left(1-\varrho(s, u)\right)du\right)ds
\end{align}
for any $h\in C_c(\mathbb{R}^3), t\leq T$ and 
\begin{equation}\label{equ 4.4}
\lim_{r\rightarrow+\infty}\sup_{x\sim y, z\sim w, s\geq 0, N\geq 1}{\rm Cov}_{\nu_{\rho, N}}\left(\left(\eta_{s}(y)-\eta_{s}(x)\right)^2, \left(\eta_{s+r}(z)-\eta_{s+r}(w)\right)^2\right)=0.
\end{equation}
The proof of \eqref{equ 4.4} in the homogeneous case given in \cite{Xue2026} does not rely on the homogeneous assumption and hence Equation \eqref{equ 4.4} holds for general $\rho$. Now we only need to check \eqref{equ 4.3}. According to the definition of $\varrho$ and the Donsker's invariance principle of the simple random walk, we have
\begin{equation}\label{equ 4.5}
\lim_{N\rightarrow+\infty}\sup_{x: \|x\|_1\leq K\sqrt{N}}\left|\tilde{\mathbb{E}}\rho\left(X_{sN}^x/\sqrt{N}\right)-\varrho(s, x/\sqrt{N})\right|=0
\end{equation}
for any $K>0$ and $s>0$. By \eqref{equ 4.5} and an argument similar to that leading to \eqref{equ 3.8}, we have the following analogue of \eqref{equ 3.8},
\begin{align}\label{equ 4.6}
&\lim_{N\rightarrow+\infty}\sup_{x, y: \|x\|_1\leq K\sqrt{N}, y\sim x}\left|\mathbb{E}_{\nu_{\rho, N}}\left(\left(\eta_{sN}(y)-\eta_{sN}(x)\right)^2\right)
-2\gamma_3\varrho(s, x/\sqrt{N})\left(1-\varrho(s, x/\sqrt{N}\right)\right|\notag\\
&=0
\end{align}
for any $K>0$ and $s>0$. Since each $x\in \mathbb{Z}^3$ has $6$ neighbors, Equation \eqref{equ 4.3} follows from \eqref{equ 4.6} and then the proof is complete.
\qed

Now we prove Lemma \ref{lemma 4.2}.
\proof[Proof of Lemma \ref{lemma 4.2}]
By \eqref{equ 4.2}, for any $t\leq T$, we have
\[
\frac{1}{N^{\frac{3}{4}}}\mathcal{M}_{tN}^{tN,N}=\mathcal{Y}^N(b_t^N),
\]
where
\[
b_t^N(s, u)=\sqrt{N}v\left(N(t-s), \sqrt{N}u_N\right)1_{\{s\leq t\}}
\]
for any $u\in \mathbb{R}^3$ and $0\leq s\leq T$. Let $b_t$ be defined as in Section \ref{section two}. It is shown in \cite{Xue2026} that $b_t^N$ converges to $b_t$ in $L^2([0, T]\times \mathbb{R}^3)$ as $N\rightarrow+\infty$. Therefore, by Lemma \ref{lemma 4.2}, 
\[
\frac{1}{N^{\frac{3}{4}}}\left(\mathcal{M}_{t_1N}^{t_1N, N}, \mathcal{M}_{t_2N}^{t_2N, N}, \ldots, \mathcal{M}_{t_mN}^{t_mN, N}\right)
\]
converges weakly to $\left(\mathcal{Y}(b_{t_1}), \ldots, \mathcal{Y}(b_{t_m})\right)$ as $N\rightarrow+\infty$.  According to the definitions of $\mathcal{Y}$ and $b_t$, 
$\left(\mathcal{Y}(b_{t_1}), \ldots, \mathcal{Y}(b_{t_m})\right)$ is a Gaussian random vector such that
\[
{\rm Cov}\left(\mathcal{Y}(b_{t_i}), \mathcal{Y}(b_{t_j})\right)=
12\gamma_3\int_0^{t_i}\int_{\mathbb{R}^3}b_{t_i}(s,u)b_{t_j}(s, u)\varrho(s, u)\left(1-\varrho(s, u)\right)dsdu
\]
for any $t_i<t_j$. Therefore, $\left(\mathcal{Y}(b_{t_1}), \ldots, \mathcal{Y}(b_{t_m})\right)$ and $\sqrt{12\gamma_3}\left(\zeta_{t_1}, \ldots, \zeta_{t_m}\right)$ have the same probability distribution and the proof is complete. 
\qed

At last, we prove Theorem \ref{theorem 2.2 main d=3}.
\proof[Proof of Theorem \ref{theorem 2.2 main d=3}]
It is shown in Section 3 of \cite{Xue2026} that, when $\rho\equiv p\in (0, 1)$, there exists $K_7<+\infty$ independent of $s, t, N$ such that
\begin{equation}\label{equ 4.7}
\mathbb{E}_{\nu_{\rho, N}}\left(\left(\frac{1}{N^{\frac{3}{4}}}\xi_{tN}^{N, 3}-\frac{1}{N^{\frac{3}{4}}}\xi_{sN}^{N, 3}\right)^2\right)\leq K_6(t-s)^{\frac{3}{2}}
\end{equation} 
for any $N\geq 1$ and $0\leq s\leq t$. The proof of \eqref{equ 4.7} given in \cite{Xue2026} does not rely on the homogeneous assumption and hence \eqref{equ 4.7} holds for general $\rho$. By \eqref{equ 4.7} and Corollary 14.9 of \cite{Kallenberg1997}, 
\[
\left\{\frac{1}{N^{\frac{3}{4}}}\xi_{tN}^{N, 3}:~0\leq t\leq T\right\}_{N\geq 1}
\]
are tight under the uniform topology. Theorem \ref{theorem 2.2 main d=3} follows from Equation \eqref{equ 4.1}, Lemmas \ref{lemma 4.1}, \ref{lemma 4.2} and the above tightness. 
\qed

\quad

\textbf{Acknowledgments.}
The author is grateful to financial
supports from the National Natural Science Foundation of China with grant numbers 12371142 and 12271026.

{}
\end{document}